\documentclass{amsart}

\usepackage{amssymb,amsmath,amsthm,epsfig}
\usepackage{comment,url}
\usepackage{graphicx}
\usepackage{graphicx}
\usepackage{caption}
\usepackage{subcaption}
\newtheorem{lemma}{Lemma}[section]
\newtheorem{theorem}{Theorem}[section]

\newtheorem{alg}{Algorithm}[section]

\theoremstyle{remark}
\newtheorem{remark}{Remark}[section]

\newtheorem{example}[theorem]{Example}

\newcommand{\re}{\mathbb{R}}

\newcommand{\N}{\mathbb{N}}

\newcommand{\st}{\mbox{subject to}}
\newcommand{\reff}[1]{(\ref{#1})}

\newcommand{\mc}[1]{\mathcal{#1}}

\newcommand{\bbm}{\begin{bmatrix}}
	\newcommand{\ebm}{\end{bmatrix}}
\newcommand{\bmx}{\begin{matrix}}
	\newcommand{\emx}{\end{matrix}}
\renewcommand{\theequation}{\thesection.\arabic{equation}}
\numberwithin{equation}{section}

\begin{document}

\title[WSM for nonlinear constrained optimization]{A New Working Set Method for Nonlinear 
Inequality Constrained Minimization}

\author[Suhan Zhong]{Suhan Zhong}
\author[Jianxin Zhou]{Jianxin Zhou}
\address{Suhan Zhong and Jianxin Zhou, Department of Mathematics, Texas A\&M University,
        College Station, TX 77843}
\date{}
\begin{abstract}
The main purpose of this project is to develop a new active set method (ASM) called a working set method (WSM) 
for solving a general nonlinear inequality constrained minimization problem in a Hilber space. Mathematical analysis 
is carried out to validate the method and to show its merits over other ASMs. Since the method is quite general and new, 
some results on its implementation and application are presented. Numerical examples on  some benchmark problems 
are carried out to test the method.
\end{abstract}
\maketitle

\section{Introduction}
\setcounter{equation}{0}
\renewcommand{\theequation}{\thesection.\arabic{equation}}

Let $H$ be a Hilbert space with inner product $\langle\cdot,\cdot\rangle$ and its norm $\|\cdot\|$.
Consider a general nonlinearly constrained minimization problem
\begin{equation}\label{EQ1.0}
\left\{\begin{array}{cl}
\underset{u\in H}{\mbox{Minimize}} & J(u)\\
\mbox{subject to} & g_i(u)\le 0,\, i \in \mc{I}_1,\\
& h_i(u) = 0,\, i\in \mc{I}_2,
\end{array}
\right.
\end{equation}
where $J, g_i, h_j: H\rightarrow \re$ are $\mc{C}^1$ functions,
the label sets $\mc{I}_1,\,\mc{I}_2$ are finite or empty. 

In this paper, we develop a new numerical method for solving the nonlinear optimization problem \reff{EQ1.0}.
It is well-known that under some constraint qualifications,
every optimizer $u$ of \reff{EQ1.0} satisfies the following Karush-Kuhn-Tucker (KKT) conditions: 
\begin{eqnarray}
 \mbox{\rm (Stationarity)} &
-J'(u)=\sum\limits _{i\in\mc{I}_1}\mu _{i}g'_{i}(u)+\sum\limits _{i\in\mc{I}_2}\lambda _{j}h'_{j}(u),\label{EQKKT1}\\
 \mbox{\rm (Primal feasibility)}  &  g_{i}(u)\leq 0,\, i\in \mc{I}_1,\quad h_{i}(u)=0,\, i\in\mc{I}_2,\label{EQKKT2}\\
  \mbox{\rm (Dual feasibility)} & \mu _{i}\geq 0,\, i\in \mc{I}_1,\label{EQKKT3}\\
 \mbox{\rm (Complementarity)} & \mu _{i}g_{i}(u) = 0,\, i\in \mc{I}_1.\label{EQKKT4}
\end{eqnarray}
where $\lambda_i, \mu_i$ are called respectively {\it Lagrange} and {\it KKT} multipliers.
Every $u\in H$ that satisfies \reff{EQKKT1}-\reff{EQKKT4} is called a KKT point of \reff{EQ1.0}.
It is very challenging to solve the KKT system \reff{EQKKT1}-\reff{EQKKT4} directly.
Indeed, only in few special cases where a closed-form solution can be derived analytically.
In general, many optimization algorithms can be interpreted as methods for numerically solving the KKT system 
(see e.g., [\ref{BOY}]). This motivates us to design an algorithm to find KKT points of \reff{EQ1.0}.
It is understood that using an objective function descent direction for searching a KKT point will mostlikely lead to a 
solution of (\ref{EQ1.0}).
 
As a class of numerical methods for solving nonlinear constrained minimization problems, the active set methods (ASMs) 
identify the active constraints in a set of inequality constraints and treat them as equality constraints, thereby convert an 
inequality-constrained problem into a simpler equality-constrained subproblem. Comparing to other methods,  ASMs have 
some distinctive features, e.g., the ability to warm-start with a good initial point.They are quite successful in (sequential) 
quadratic programming (QP) methods \cite{GMS1,GMS2,GR1,GR2,GR3,NW}. However, their weaknesses are also obvious.
Early ASMs \cite{ROS1,ROS2} assume that the optimal active set $I_A(u^*) := \{i\in\mc{I}_1: g_i(u^*) = 0\}$ is known 
for a local optimizer $u^*$ of \reff{EQ1.0}. Such a method lacks of flexibility to adaptively update the active index set when 
$I_A(u^*)$ is unknown and in particular, nonlinear inequality constraints are involved. That is, in the correction process in 
[\ref{ROS2}], the working index set used to form a correction basis is equal to the active index set. Then the correction 
process will not be superlinear and the active index set cannot be reduced. To resolve this problem, a commonly applied 
strategy is to use the most negative KKT multiplier component \cite{NW} to identify and remove a false active index from 
the active index set. This process will not be superlinear and still lacks of dynamics to update the active index set.
Also this process puts a natural lower bound on the number of iterations to reach an optimal solution and thus can be very slow.
As a result, ASMs are only used in small-median-scale QP problems in most early works in the literature. In order for an ASM 
to be more efficient in solving larger scale problems, most recent ASMs in the literature focus on predicting, 
estimating or approximating such an ``optimal'' active index set \cite{CHR,FGW}.

The main purpose of this project is to develop a new ASM called a working set method (WSM) for solving \reff{EQ1.0}.
Without loss of generality, we assume $\mc{I}_2 = \emptyset$, since equality constraints can always be viewed as 
fixed active inequality constraints. Let $\mc{I} = \mc{I}_1\cup\mc{I}_2$. Then \reff{EQ1.0} is simplified to be
\begin{equation}\label{EQ1.1}
\left\{\begin{array}{cl}
\underset{u\in H}{\mbox{Minimize}} & J(u)\\
\mbox{subject to} & g_i(u)\le 0\,, i \in \mc{I}.
\end{array}
\right.
\end{equation}
Let $\Omega$ be the feasible set of (\ref{EQ1.1}). For each $u\in \Omega$, we denote the {\it active index set} (AIS) by
\begin{equation}\label{eq:I_A(u)}
I_A(u) \,:=\, \{ i\in \mc{I} : g_i(u)=0\}.
\end{equation}
Throughout this paper, we assume the {\it linearly independent constraint qualification} (LICQ) holds: 
for each $u\in\Omega$, vectors  $\{g'_i(u): i\in I_A(u)\neq\emptyset\}$ are linearly independent.
The KKT system \reff{EQKKT1}-\reff{EQKKT4} can be simplified if the AIS is known for a KKT point.
Suppose $u$ is a KKT point of \reff{EQ1.1}. By the complementary slackness (\ref{EQKKT4}),  we have 
$\mu_i=0$ for $i\in \mc{I}\setminus I_A(u)$. Then stationarity (\ref{EQKKT1}) can be written as
\begin{equation}\label{EQKKT1-b}
-J'(u) = \sum _{i\in \mc{I}} \mu _{i}g'_{i}(u)=
\sum_{i\in I_A(u)} \mu_i g'_{i}(u),
\end{equation}
and the dual feasibility (\ref{EQKKT3}) becomes
\begin{equation}\label{EQ1.8}
\mu_i\ge 0,\quad \forall\, i\in I_A(u).
\end{equation}

In a nonlinear search for a KKT point, a computed solution candidate needs to be in the feasible region, 
thus various gradient projections and correction techniques are proposed in the literature. For example,
the conditional projection method \cite{FW} and its variations require the feasible set to be convex.
So they are only applicable for linear equality and convex inequality constrained optimization problems.

Let $C\subset H$ be a closed convex set and $\mathbb{P}_{C}(v)$ be the projection of $v\in H$ onto $C$.
We need some classical convex projection results.
\begin{lemma}\label{LM0.1}[\ref{BOY}] 
Let $C\subset H$ be a closed convex cone and $u\in H$. We have
\[\begin{array}{l}
\langle v-{\mathbb P}_C(u), u-{\mathbb P}_C(u)\rangle \le 0,\quad \forall v\in C;\\
\langle {\mathbb P}_C(u), u-{\mathbb P}_C(u)\rangle=0\quad \mbox{\rm and}\quad
\langle v, u-{\mathbb P}_C(u)\rangle\le 0, \forall v\in C;\\
\langle u, {\mathbb P}_C(u)\rangle=\|{\mathbb P}_C(u)\|^2\quad \mbox{\rm and}\quad 
\langle u, u-{\mathbb P}_C(u)\rangle=\| u-{\mathbb P}_C(u)\|^2.
\end{array}\]
\end{lemma}

For $u\in\Omega$, the AIS $I_A(u)$ is defined only by constraint functions and is independent of a search direction 
used to minimize the objective function $J$. While as numerical algorithms are concerned, a more active constrained 
boundary point of $\Omega$ can always be approximated by less or even none actively constrained points in $\Omega$.
In this case, a search direction with faster descending on objective function values can be expected with more freedoms. 
In other words, the smaller an AIS is, the more effective the projected gradient is to descend the objective function values,
but the more possible the feasibilities are to be violated. Thus it is a quite interesting problem to find a minimum AIS 
while maintaining all the feasibilities during computations.  In order to reduce active constraints and make a correction 
process superlinear, we introduce the notion of a {\em false active index} defined by a search direction. 
Such a false active constraint can actually be treated as an inactive constraint, therefore be excluded from a correction basis. 
With such a notion, the AIS $I_A$ can be effectively reduced to form a working index set $I_W$ while maintain all the feasibilities.

By a brief description of our original idea, we develop a new {\it working set method} (WSM), a search-and-correction process, 
by introducing two key terms for each feasible point $u\in \Omega$. One is the smallest {\it working index set} (WIS) 
\begin{equation}\label{EQ1.9}
I_W(u)=\{ i\in I_A(u) : \langle d_W(u), g'_i(u)\rangle =0\rangle \},
\end{equation} 
which is used to form a basis for a correction process to maintain the feasibilities. 
The other is the corresponding projected gradient
\begin{equation}\label{eq:d_W}
d_W(u) = -J'(u)-{\mathbb P}_{C_W(u)}(-J'(u))
\end{equation}
as a search direction,  where $C_W(u)$ is the convex polyhedral cone generated by the edges $g'_i(u), i\in I_W(u)$. 
Unfortunately, an infinite loop is formed in the definitions (\ref{EQ1.9}) and (\ref{eq:d_W}), i.e., $I_W(u)$ and $d_W(u)$ 
depend on each other. This infinite loop will be broken only after the equality (\ref{EQDA=DW}) is established.
Then we can see that an indix $i\in I_A(u)\setminus I_W(u)$ satisfies 
\begin{equation}\label{EQ1.11}
\langle d_W(u), g'_i(u)\rangle<0
\end{equation} 
and will be identified as {\em a false active indix} (FAI). Such a FAI is in opposite to the original intention of a correction 
process which is designed to pull an infeasible point back to the feasible set $\Omega$, not to push a less active constrained 
feasible point to a more active constrained position, see Figure~\ref{FIG0}. Also due to the strict inequality in (\ref{EQ1.11}), 
a correction process containing such a FAI in its basis will not be superlinear. It will cause difficulty to establish a stepsize rule 
and in algorithm convergence analysis as well. Thus it should be removed from a correction process.
\begin{figure}[htb]
\vspace{-0.0in}
\centerline{
\hspace{0.0in}
\epsfig{file=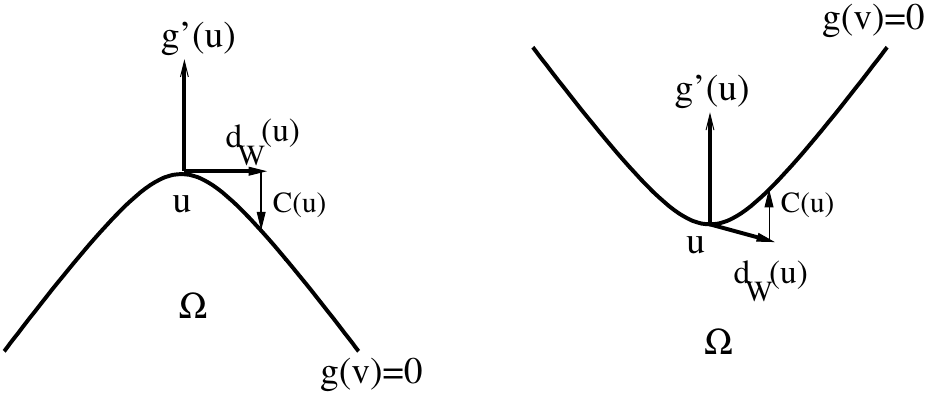,height=1.6in,width=4.0in}
}
\caption{$C(u)$ represents a correction process. On the left, with an active index, 
$C(u)$ pulls the infeasible point $u+td_W(u)$ back to a feasible point $u+td_W(u)+C(u)$ in $\Omega$;
on the right, with a false active index, $C(u)$ pushes a less active constrained point $u+td_W(u)$ to 
a more active constrained point in $\Omega$.}
\label{FIG0}
\end{figure}
After those FAIs being removed from a correction basis, we make a correction by solving proper functions $c_i(t)$ 
with $i\in I_W(u)$ for small $t>0$ such that
\[ u(t)=u+td_W(u)+\sum_{i\in I_W(u)}c_i(t)g'_i(u)\in  \Omega. \]
The implicit function theorem (IFT) guarantees the existence of such $c_i$'s with $c_i(t)=o(t)$, or superlinear. 
Then certain stepsize rule will be established to guide the algorithm to determine a stepsize $t$ and to update 
to a new point $u(t)$. 

This paper is organized as follows.
In Section~\ref{SEC2}, 
we introduce the working index set and its correctable steepest descent direction (CSDD),
and present some basic properties to validate our new WSM.
The flow chart of WSM is given in Section~\ref{SEC3}.   
In Section~\ref{SEC4}, we discuss some interesting implementations and applications of our WSM. 
In Section~\ref{SECEX}, we carry out numerical experiments of the new algorithm on some benchmark examples to nonlinear 
inequality constrained optimization and present their numerical results.
As a conclusion, in Section~\ref{SEC5}, analysis is carried out to compare our WSM with ASMs in the literature. 
A global convergence of the algorithm will be established in a subsequent paper [\ref{Zhou0}].

\section{The working index set and its correctable steepest descent direction}\label{SEC2}
\setcounter{equation}{0}
\renewcommand{\theequation}{\thesection.\arabic{equation}}

We focus on the nonlinear optimization probelm \reff{EQ1.1}. With the introduction of AIS $I_A(u)$, 
its KKT system can be equivalently written as
\begin{equation}\label{eq:KKT:equiv}
-J(u') = \sum\limits_{i\in I_A(u)} \mu_i g'_{i}(u),\quad 
\mu_i\ge 0\,\forall i\in I_A(u).
\end{equation}
For $u\in\Omega$, denote the polyhedral cone and the subspace determined by $I_A(u)$
\begin{eqnarray}
\label{EQ2.0}
C_A(u) \,=\, \mbox{cone}(\{g_i'(u): i\in I_A(u)\}) \,:=\, \Big\{ \sum_{i\in I_A(u)}  a_ig'_i(u) : a_i\ge 0 \Big\},\\
\label{eq:S_A(u)}
S_A(u) \,=\, \mbox{span\,}\{g'_i(u):  i\in I_A(u)\}\,:=\, \Big\{ \sum_{i\in I_A(u)}  b_ig'_i(u) : b_i\in\re \Big\}.
\end{eqnarray} 
For the special case that $I_A(u) = \emptyset$, we simply set $C_A(u)=S_A(u)=\{0\}$.
From \reff{eq:KKT:equiv}, we see that a point $u\in\Omega$ is KKT if and only if $-J'(u)\in C_A(u).$
Define the {\em correctable steepest descent direction} (CSDD) of $J$ to (\ref{EQ1.1}) at $u\in \Omega$ by
\begin{equation}\label{EQCSDD}
d_A(u) \,:=\, -J'(u) - {\mathbb P}_{C_A(u)}(-J'(u)).
\end{equation} 
It can be viewed as the residual of the KKT condition (\ref{eq:KKT:equiv}). Then the following equivalent relation holds.
\begin{lemma}[KKT-optimality]\label{LM1} 
A point $u\in \Omega$ is a KKT point of \reff{EQ1.1} if and only if $d_A(u) = 0$.
\end{lemma}
\begin{proof}
 $d_A(u) = 0\Leftrightarrow-J'(u) = {\mathbb P}_{C_A(u)}(-J'(u))\in C_A(u)\Leftrightarrow u \mbox{ is a KKT point}$.
\end{proof}

Define a correctable descent direction commonly used in ASMs 
\begin{equation}\label{EQ1.4}
d_S(u) \, :=\, -J'(u) - {\mathbb P}_{S_A(u)}(-J'(u)).
\end{equation}
It is clear that $d_S(u) = 0$ if and only if $-J'(u)={\mathbb P}_{S_A(u)}(-J'(u))\in S_A(u)$. By Lemma~\ref{LM1}, 
if $u\in\Omega$ is a KKT point, then
\[
-J'(u)\,\in\, C_A(u)\,\subseteq\, S_A(u)\quad \Rightarrow\quad d_S(u) = 0.
\]
But $d_S(u)=0$ does not necessarily imply that $u$ is a KKT point. This motivates us to use $d_A(u)$ instead of $d_S(u)$ 
as a search direction to design a descent algorithm for  solving \reff{EQ1.1}: a feasible sequence $\{u^{(k)}\}$ will be 
generated by the algorithm and converge to a KKT point if $d_A(u^{(k)})\rightarrow 0$ as $k\rightarrow \infty$.
\begin{lemma}\label{LM2}
If $u\in\Omega$ is not a KKT point, 
then $d_A(u)\neq 0$ and all the following statements hold:
\begin{enumerate}
\item  $\langle v - {\mathbb P}_{C_A(u)}(-J'(u)) , \, d_A(u) \rangle\le 0,\;\forall v\in C_A(u)$. In particular,  
\[ \langle g'_i(u) - {\mathbb P}_{C_A(u)}(-J'(u)) , \, d_A(u) \rangle\le 0,\;  \forall i\in I_A(u);\hfill\]
\item $\langle d_A(u),\, {\mathbb P}_{C_A(u)}(-J'(u)) \rangle = 0;$
\item $\langle v, d_A(u)\rangle\le  0,\;\forall v\in C_A(u)$. In particular,
\[ \langle g'_i(u),\, d_A(u)\rangle\le 0,\quad  \forall i\in I_A(u);\hfill\]
\item $\langle -J'(u),\,  d_A(u)\rangle \,=\, \|d_A(u)\|^2 \,=\, \|-J'(u)\|^2 - \|{\mathbb P}_{C_A(u)}(-J'(u))\|^2 >0;$
\item $\langle {\mathbb P}_{C_A(u)}(-J'(u)),\,  -J'(u) \rangle \,=\, \|{\mathbb P}_{C_A(u)}(-J'(u))\|^2.$
\end{enumerate}
\end{lemma}
\begin{proof}
We have $d_A(u) = -J'(u) - {\mathbb P}_{C_A(u)}(-J'(u))$ by \reff{EQCSDD}. Since $C_A(u)$ is a closed convex cone, 
all these conclusions are followed directly by Lemma~\ref{LM0.1}.
\end{proof}
Suppose $u\in\Omega$ is not a KKT point. An index $i\in I_A(u)$ is called a {\em false active index (FAI)} 
if there exists $t_0>0$ such that
\[
g_i(u+td_A(u)) <0\quad \mbox{for every}\quad 0<t\le t_0.
\] 
Such a FAI should be treated as an inactive index and be excluded from a correction process. To identify FAIs, 
we define the working index set (WIS) at $u\in \Omega$ as
\begin{equation}\label{EQWIS}
I_W(u) \,:=\, \{ i\in I_A(u): \langle g'_i(u), d_A(u)\rangle = 0 \}.
\end{equation} 
Then by Lemma~\ref{LM2}~(3), it holds that
\begin{equation}\label{EQ2.5}
\langle g'_i(u),\, d_A(u)\rangle < 0,\quad \forall i\in I_A(u)\setminus I_W(u),
\end{equation}
which implies for all small $t>0$, it holds
\[ g_i(u+td_A(u)) \,=\, g_i(u)+t\langle g_i'(u),d_A(u)\rangle +o(t) \,<\, 0. \]
In other words, every index in AIS but not in WIS is an FAI, which should be excluded from our correction process.
Thus we denote the {\em working polyhedral cone}
\begin{equation}\label{EQWPC}
C_W(u) \,:=\, \mbox{cone}(\{g_i'(u): i\in I_W(u)\}).
\end{equation}
It is a closed convex subset of $C_A(u)$. By Lemma~\ref{LM2}, we define a supporting hyperplane of $C_A(u)$ 
at the point ${\mathbb P}_{C_A(u)}(-J'(u))$,
\[
{\mathcal H}_A(u) \,:=\, \{ v\in H : \langle v, d_A(u)\rangle = 0\}.
\]
It separates $C_A(u)$ in its lower half space  
$\{ v\in H : \langle v, d_A(u)\rangle \le 0\}$ from $-J'(u)$ in the upper half space $\{ v\in H : \langle v, d_A(u)\rangle>0\}$. 
Note that $C_W(u)$ is also a subset of $\mc{H}_A(u)$. Indeed, $C_W(u)$ is a face of the polyhedral cone $C_A(u)$ 
whereat the projection ${\mathbb P}_{C_A(u)}(-J'(u))$ of $-J'(u)$ onto $C_A(u)$ locates.  It implies that the projections of 
$-J'(u)$ onto $C_A(u)$ and onto $C_W(u)$ are the same, i.e.,
\begin{equation}\label{EQDA=DW}
\begin{aligned}
d_W(u) &=  -J'(u)-{\mathbb P}_{C_W(u)}(-J'(u)) & \\
&= -J'(u)-{\mathbb P}_{C_A(u)}(-J'(u)) &= d_A(u).
\end{aligned}
\end{equation}
It is this equality that motivates us to use $d_A(u)$ instead of $d_W(u)$ to define $I_W(u)$ and $C_W(u)$, 
and thus breaks the infinite-loop in the definitions (\ref{EQ1.9}) and (\ref{eq:d_W}).
Now we use $d_W(u)$ and $I_W(u)$ to design a new algorithm where $d_A(u)=d_W(u)$ is used as a descent search dirction 
and the vectors $\{g'_i(u): i\in I_W(u)\}$ are used to form a basis for a superlinear correction process.
\begin{figure}[hbt]
\vspace{-0.2in}
\centerline{
\epsfig{file=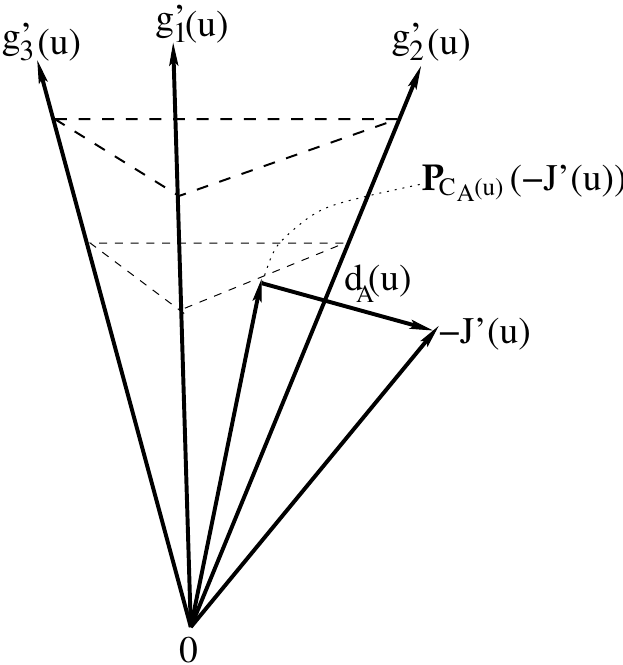,height=1.6in,width=2.0in}\hspace{0.2in}
\epsfig{file=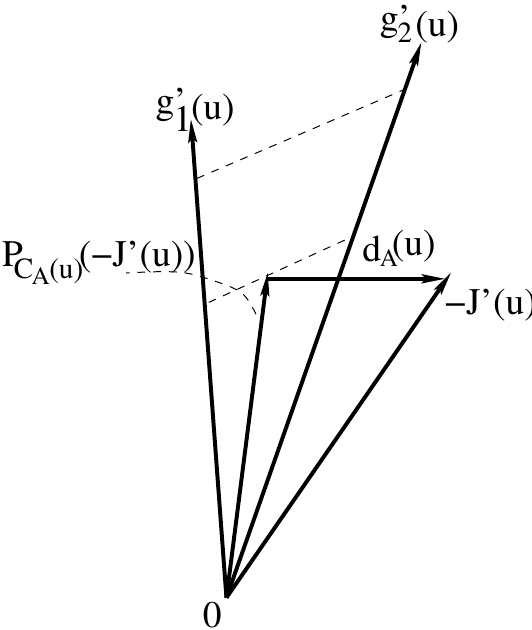,height=2.0in,width=2.0in}
}
\caption{Geometry of $d_A(u)$ for different AIS and WIS. 
On the left, $I_A(u) = \{1,2,3\}$ and $I_W(u) = \{1,2\}$; 
on the right, $I_A(u) = I_W(u) = \{1,2\}$.
For both cases, $d_A(u) = d_W(u)$ since ${\mathbb P}_{C_A(u)}(-J'(u))={\mathbb P}_{C_W(u)}(-J'(u))$.}
\label{FIG1}
\end{figure}

\begin{lemma}[Superlinear Correctability]\label{LM3}
If $u\in\Omega$ is not a KKT point,
then there exist locally ${\mathcal C}^1$ functions $c_i(t),\,i\in I_W(u)$ and a constant $t_0>0$ dependent on $u$ such that
for every $t\in (0, t_0)$, the correction
\begin{equation}\label{EQCRR1}
u(t) \,=\, u+td_W(u)+\sum_{i\in I_W(u)} c_i(t)g'_i(u)
\end{equation}
is feasible for \reff{EQ1.1} and satisfies
\begin{equation}\label{EQCRR2}
g_i(u(t))=0,\;\forall i\in I_W(u).
\end{equation}
In addition,  it holds that
\begin{equation}\label{EQ2.9}
c_{i}(t)=o(t),\;\forall i\in I_W(u)\quad\mbox{and}\quad 
\Big\|\sum_{i\in I_W(u)} c_i(t)g'_i(u)\Big\| = o(t).
\end{equation}
\end{lemma}
\begin{proof}
Without loss of generality, suppose $I_W(u) = \{1, 2,\dots, k\}$ and denote by $g = (g_1, \ldots, g_k)^T$ the tuple of working constraints. For $u(t)$ as in \reff{EQCRR1}, let $c = (c_1, \ldots, c_k)$ be a vector function of $t$ obtained from solving the system
\begin{equation}\label{EQCS}
G (t,c(t)) \,:=\, \big[ g_{1}(u(t)),\, \ldots, \, g_{k}(u(t)) \big]^T={\bf 0}_k=(0,...,0)^T\in {\mathbb R}^k.
\end{equation}
Clearly, $u(0)=u, c(0)={\bf 0}_k$ and $G(0,c(0))=g(u)={\bf 0}_k.$ Differentiate $G(t,c(t))$ with respect to $t$ at $t=0$. 
We obtain
\begin{equation}\label{eq:GtGc}
G_t(0,{\bf 0}_k)+G_c(0,{\bf 0}_k)(c'(0))={\bf 0}_k,
\end{equation}
where $G_t(0)$ and $c'(0)$ are column vectors given by
\[
G_t(0,\mathbf{0}_k) \,=\, [ \langle d_W(u),g'_{i}(u)\rangle ]_{(1\le i\le k)},\quad
c'(0) \,=\, [c_i'(0)]_{(1\le i\le k)},
\]
and $G_c(0,\mathbf{0}_k)$ is the gram matrix 
\[
G_c(0,\mathbf{0}_k) \,=\, {\rm Gram}(g'(u)) \,=\, 
\left[ \langle g'_{i}(u), g'_{j}(u)\rangle\right]_{(1\le i,j \le k)}.
\]
$G_t(0,\mathbf{0}_k) = \mathbf{0}_k$ by \reff{EQWIS} and \reff{EQDA=DW}, 
and $G_c(0,\mathbf{0}_k)$ is nonsingular under LICQC.
Then by the implicit function theorem (IFT, e.g., [\ref{KCC}, \ref{IFT}]), 
there is a unique vector function $c$ and constant $t_1(u)>0$ such that 
\[
G(t, c(t)) ={\bf 0}_k\; \forall 0<t<t_0(u)\quad \mbox{and}\quad
\mbox{$c(t)$ is locally ${\mathcal C}^1$ near $t=0$}.
\] 
By \reff{eq:GtGc}, $c'(0)= -G_c(0,{\bf 0}_k)^{-1}G_t(0,{\bf 0}_k)$ and $G_t(0,{\bf 0}_k) ={\bf 0}_k$. It follows
\begin{equation}\label{eq:c=o(t)}
c'(0) \,=\, c(0) \,=\, \mathbf{0}_k\quad \Rightarrow\quad
c_i(t)\,=\,o(t)\quad \forall i\in I_W(u).
\end{equation} 
In addition, since $|I_W(u)|\le |\mc{I}|<\infty$ and $\sum_{i\in I_W(u)} \|g'_i(u)\|<M(u)$ for some constant $M(u)>0$, 
we have
\begin{equation}\label{eq:sumc=o(t)}
\Big\|\sum_{i\in I_W(u)} c_i(t)g'_i(u) \Big\| \le\sum_{i\in I_W(u)} \|g'_i(u)\|\Big(\max_{i\in I_W(u)}|c_i(t)|\Big) = o(t).
\end{equation}
Note $g_i(u)<0\,\forall i\in \mc{I}\setminus I_A(u)$ and 
$g_i(u)=0, \langle g'_i(u), d_W(u)\rangle <0\, \forall i\in I_A(u)\setminus I_W(u)$. 
Then by properties \reff{eq:c=o(t)}--\reff{eq:sumc=o(t)},  there is $t_2(u)>0$ such that $g_i(u(t))< 0$ for all 
$i\in \mc{I}\setminus I_W(u)$ and $0<t<t_2(u)$. Denote $t_0=\min\{t_1(u), t_2(u)\}$.  When $0<t<t_0$, $u(t)$ is feasible.
Therefore, all the conclusions are proved.
\end{proof}
The superlinear property in the correction process is critical for us to maintain the feasibility condition 
$g_i(u(t))\le 0$ for $i\in{\mathcal I}\setminus I_W(u)$. 

\begin{lemma}[Stepsize Rule for CSDD]\label{LM2.5} 
Suppose $u\in \Omega$ is not a KKT point. Then there exists a scalar $t_0>0$ dependent on $u$ such that
\begin{equation}\label{EQSR}
J(u(t)) - J(u) <-\frac{t}{2}\|d_W(u)\|^2,\quad \forall\, 0<t<t_0,
\end{equation}
where $u(t)$ is defined as in \reff{EQCRR1} satisfying \reff{EQCRR2}-\reff{EQ2.9}.
\end{lemma}
\begin{proof}
By Lemma~\ref{LM3}, there exists $\tau_1>0$ such that $u(t)$ satisfies \reff{EQCRR2}-\reff{EQ2.9} 
for every $0<t<\tau_1$ with unique functions $c_i(t),\;i\in I_W(u)$. Then \reff{EQ2.9} implies that
\begin{equation}\label{EQ2.20}
u(t) - u \,=\, td_W(u)+\sum_{i\in I_W(u)}c_i(t)g'_i(u) \,=\, O(t).
\end{equation}
Since $J$ is a $\mc{C}^1$ function, 
there exists $\tau_2>0$ such that for every $0<t<\tau_2$, 
\[\begin{aligned}
J(u(t))-J(u) &= \langle J'(u), u(t)-u\rangle+o(\|u(t)-u\|)\\
&=  t\langle J'(u), d_W(u)\rangle + \sum\limits_{i\in I_W(u)}c_i(t)\langle J'(u), g'_i(u)\rangle+o(t)\\
&= t\langle J'(u), d_W(u)\rangle + o(t).
\end{aligned}\]
In the above, the last equality is obtained by \reff{EQ2.9} and Cauchy-Schwartz inequality.
Recall from \reff{EQDA=DW} and Lemma~\ref{LM2}~(2), we have
\[
J'(u) = -d_W(u) - \mathbb{P}_{C_A(u)}(-J'(u)),\quad  d_A(u)\perp \mathbb{P}_{C_A(u)}(-J'(u)).
\]
Since $d_W(u) = d_A(u)$, we have
\[
\langle J'(u), d_W(u)\rangle \,=\, \langle -d_A(u), d_W(u)\rangle
\,=\,-\|d_W(u)\|^2.
\]
Then for every $0<t<\tau_2$, we obtain
\[
J(u(t)) - J(u) = -t\|d_W(u)\|^2+o(t)<-\frac{t}{2}\|d_W(u)\|^2.
\]
Finally the conclusion holds for $t_0 = \min\{\tau_1, \tau_2\}$.
\end{proof}

\begin{remark}
In the stepsize rule, we may assume $\|d_W(u)\|$ is bounded for all $u\in \Omega$. 
Otherwise, $d_W(u)$ can be replaced by its normalization, i.e., $d_W(u)/C(u)$ where $C(u)=\max\{1, \|d_W(u)\|\}$. 
Then (\ref{EQSR}) becomes 
\[ J(u(t))-J(u) \,<\, -\frac{t}{2C(u)}\|d_W(u)\|^2. \]
\end{remark}

\section{The correctable steepest descent method}\label{SEC3}
\setcounter{equation}{0}

We present a new WSM algorithm, also called the {correctable steepest descent method} (CSDM). 
\begin{alg}[CSDM algorithm]\label{alg:WSM} 
For the constrained optimization problem \reff{EQ1.1},
let $\tau>0$, $\varepsilon>0$, $u^{(0)}\in \Omega$, set $k=0$, and do the following:
\begin{description}
\item[Step 1]
Determine AIS $I_A =\{ i\in \mc{I}: g_i(u^{(k)}) = 0\}$ and evaluate $g'_i(u^{(k)})$ for all $i\in I_A$.

\item[Step 2]
Compute CSDD 
$\;
d^{(k)} \,=\, -J'(u^{(k)}) - {\mathbb P}_{C_A}(-J'(u^{(k)})),
$
where $\mathbb{P}_{C_A}(\cdot)$ is the projection onto $C_A = cone\{g_i'(u^{(k)}): i\in I_A\}$.\\
If $\|d^{(k)}\|<\varepsilon$, then output $u^*=u^{(k)}$ and stop,
else go to the next step.

\item[Step 3]
Determine WIS $I_W = \{ i\in I_A : \langle d^{(k)},  g'_i(u^{(k)}) \rangle = 0 \}$. 

\item[Step 4]
Find the maximum $\bar{t}\in (0, \tau)$ such that $\bar{u} := u^{(k)}+\bar{t}d^{(k)}$ satisfies
\[
J(\bar{u}) - J(u^{(k)}) < -\frac{\bar{t}}{2}\|d^{(k)}\|^2\quad
\mbox{and}\quad
g_i(\bar{u})<0\, \forall i\in\mc{I}\setminus I_A.
\]

\item[Step 5]
Construct the correction
\begin{equation}\label{EQSSR2}
 u(t) := u^{(k)} + t d^{(k)}+\sum_{i\in I_W} c_i(t)g'_i(u^{(k)}),
 \end{equation} 
where each $c_i(t)$ is solved from the system $g_i(u(t))=0\,(i\in I_W)$.

\item[Step 6]
Use $\bar{t}$ as an initial guess to find the maximum $t_k\in (0, \tau)$ such that
\begin{equation}\label{EQSSR}
J(u(t_k)) - J(u^{(k)}) < -\frac{t_k}{2}\|d^{(k)}\|^2\quad \mbox{and}\quad
u(t_k)\in \Omega.
\end{equation}
Update $u^{(k+1)} := u(t_k)$, set $k :=k+1$ and go back to Step 2.
\end{description}
\end{alg}

\begin{remark}
In Step~2, for many cases in infinite dimensional spaces, to obtain $J'(u^{(k)})$, 
one needs to solve a differential equation by an approximation method, 
such as a FEM, a FDM, a BEM, or a spectral method, etc. 

The numerical implementation of finding the projection $\mathbb{P}_{C_A}(\cdot)$
is a typical application of our new method to a quadratic objective function subject to
 a quite simple positive cone constraint, see as in Subsections~\ref{SEC4.1} and \ref{SEC4.2}.

Step 3 is a key element. 
It allows the algorithm search to leave some AIS in the last iteration.
In numerical computations, the active set is determined up to a precision, i.e., 
$i\in I_W$ if $|\langle d^{(k)}, g_i'(u^{(k)})\rangle |<10^{-6}$, due to round-off errors.

The pre-correction strategy in Step 4 is based on Lemmas~\ref{LM3}-\ref{LM2.5} where the last higher-order term $o(t)$ in (\ref{EQCRR1}) has been omitted to save computational cost in the correction process carried out in Step 5.

In Step 5, since $\{g'_i(u^{(k)}): i\in I_W\}$ is known and linearly independent,
a Newton method is preferred to solve $c_i(t)$ from $g_i(u(t)) = 0$ for $i\in I_W$. 
Also in the implementation, both the heuristic stepsize $\bar{t}$ and the actual stepsize $t_k$ are not necessary 
to be the maximum. They can be the first value of the form $2^{-j}\tau$ to satisfy the pre-correction and
 the correction requirements respectively.
\end{remark}

\section{Some Special Cases}\label{SEC4}
\setcounter{equation}{0}
\renewcommand{\theequation}{\thesection.\arabic{equation}}
Since our problem setting and the algorithm framework are quite general, in this section, 
we discuss some special cases and their efficient implementations in Algorithm~\ref{alg:WSM}.

\subsection{Linear constraints}\label{SEC4.0}
 
Linear constrained optimziation problems are relatively simple but still an interesting type of constrained optimization problems.
For convenience, we denote 
$\mc{I} := \{1, \ldots, m\}$ 
and the constraint tuple 
$g = (g_1, \ldots, g_m)$ in $u$.
Assume \reff{EQ1.1} has all linear constraints, i.e., 
\[
g_i(u) = \langle a_i, u\rangle + b_i,\quad i\in\mc{I},
\] 
where $a_i\in H$ and $b_i\in {\mathbb R}$ are given. It is clear that $g_i'(u) = a_i$ is independent of $u$.
Let $u(t) = u+td_W(u)$. For each $i\in \mc{I}$,
\[
g_i(u(t)) \,=\, \langle a_i, u+td_W(u)\rangle + b_i 
\,=\, g_i(u) + t\langle a_i, d_W(u)\rangle.
\]
Let $I_W(u)$ be as in \reff{EQWIS}. Interestingly, for each small $t>0$, $u(t)$ is feasible for \reff{EQ1.1} 
with the unchanged working index set $I_W(u)$. To see this, for each $i\in I_W(u)$, we have $g_i(u(t))=0$ because
$g_i(u) =0$ and $\langle g'_i(u), d_W(u) \rangle = \langle a_i, d_W(u)\rangle = 0$.
On the other hand, for each $i\in {\mathcal I}\setminus I_W(u)$, we have $g_i(u(t))<0$ for small $t>0$ 
because either $i\in {\mathcal I}\setminus I_A(u)$, i.e., $g_i(u)<0$ or $i\in I_A(u)\setminus I_W(u)$, i.e.,  
$g_i(u) = 0, \langle g'_i(u), d_W(u)\rangle=\langle a_i, d_W(u)\rangle <0$.
It implies that in \reff{EQCRR1}, each $c_i(t)$ is identically zero,
thus \reff{EQCRR1} is reduced to $u(t)=u+td_W(u)$. In other words, no correction is necessary.

\subsection{Optimize over nonnegative orthant}\label{SEC4.1}

Let $H = \re^n$ and $J\in {\mathcal C}^1({\mathbb R}^n, {\mathbb R})$.
Consider the optimization problem
\begin{equation}\label{eq:nnopt}
\underset{x = (x_1, \ldots, x_n)\in\re^n}{\mbox{Minimize}}J(x)\quad
\mbox{subject to }\; x \ge 0.
\end{equation}
Its KKT-condition is: 
$x^* = (x_1^*, \ldots, x_n^*) \ge 0$, and
\begin{equation}\label{EQ6.2}
[ J'(x^*)]_i=  \frac{\partial J(x^*)}{\partial x_i} = 0\;\;\mbox{if $x_i^*>0$},\;\;
 [ J'(x^*)]_i= \frac{\partial J(x^*)}{\partial x_i} \ge 0\;\; \mbox{ if $x_i^*=0$}.
\end{equation}  
Denote $[r]^+=\max\{r,0\}$ for a real number $r$. Then in a fixed point form, \reff{EQ6.2} is equivalent to 
\begin{equation}\label{EQ4.3}
x^* \,=\, [x^*-\alpha J'(x^*)]^+,
\end{equation}
where $\alpha$ is any fixed positive scalar, and the vector in the right hand side is formed by 
$[x^*-\alpha J'(x^*)]_j^+$ for the $j$th entry. This gives a simple iterative scheme:
\begin{equation}\label{EQ6.3}
x^{(k+1)} \,=\,  [x^{(k)} - \alpha J'( x^{(k)} )]^+,\quad \forall\, k\in \N.
\end{equation}

Apply our WSM to the optimization problem \reff{eq:nnopt}. Let $e_1,...,e_n$ be the canonical basis for 
${\mathbb R}^n$ and ${\mathcal I}=\{1,...,n\}$. Then in this case, we have $g_i(x)=-x_i$, $g_i'(x) = -e_i$ 
for all $i\in {\mathcal I}$, thus all $g_i'(x)$ are orthogonal, and 
\[
I_A(x) = \{ i\in {\mathcal I}: x_i=0\},\quad C_A(x) = cone(\{-e_i: i\in I_A(x)\}).
\]
With the orthogonality of $g'_i(x)$, the projection of $-J'(x)$ onto $C_A(x)$ has the form
\[
\sum_{i\in I_A(x)} [\langle J'(x), e_i\rangle]^ + (-e_i)=-\sum_{i\in I_A(x)} [J'(x)_i]^+e_i.
\]
For convenience, for fixed $u\in\Omega$ and every vecor $v=(v_1,...,v_n)^T\in H$, we denote
\begin{equation}\label{eq:vec^p}
\mathbb{P}^-(v) = \sum_{i\in I_A(u)} [v_i]^ + e_i,\quad
v^\mathbb{P} \,:=\, v - \mathbb{P}^-(v).
\end{equation}
Then by \reff{EQCSDD}, we get
\begin{equation}\label{eq:d_W=J^p}
d_W(x) = [-J'(x)]^{\mathbb{P}},
\end{equation}
whose $i$th entry is computed by
$$
[d_W(x)]_i =\left\{ 
\begin{array}{ll}
0, & \mbox{if $x_i = 0$ and $[-J'(x)]_i<0$}, \\

[-J'(x)]_i, & \mbox{otherwise}.
\end{array}
\right.
$$
The simple positive cone constraint is a special case of linear constraints, 
so no correction is necessary in each algorithm iteration by previous discussions. 
In this case, our algorithm iteration scheme becomes
\begin{equation}\label{eq:lcx^k}
x^{(k+1)} = x^{(k)} + \alpha_k [-J'(x^{(k)})]^{\mathbb P}= 
x^{(k)} + \alpha_k [-J'(x^{(k)})-{\mathbb P}^-(-J'(x^{(k)}))],
\end{equation}
where $\alpha_k>0$ is selected to satisfy our stepsize rule and the feasibility $x^{(k+1)}\ge 0$. 

\subsection{Optimize with quadratic objective function}\label{SEC4.2}

As a classic nonlinear optimization problem, we consider to minimize a quadratic objective functional subject to a conic constraint
\begin{equation}\label{EQLQ}
\underset{{x\in {\mathbb R}^n}}{\mbox{Minimize}}\,\,\Big\|\sum_{i=1}^n x_iu_i-u_0\Big\|^2\quad
\mbox{subject to }\; x\ge 0,
\end{equation}
where $u_0, u_1,...,u_n\in H$ are given and $u_1,...,u_n$ are linearly independent.
This problem can be solved directly by Algorithm~\ref{alg:WSM}
with the iteration scheme \reff{eq:lcx^k}. It can also be solved in a more computationally convenient way. 
Note that this problem can be viewed as to compute the projection of $u_0$ onto the convex cone
$C=\mbox{cone}\{u_1,...,u_n\}$. 

\begin{lemma}\label{LM0.2}
Let $S\subset H$ be a closed subspace and $C\subset S$ be a closed convex set. For any $u\in H$, 
we have ${\mathbb P}_C({\mathbb P}_S(u))={\mathbb P}_C(u)$.
\end{lemma}
\begin{proof}
Write $u=u_S+u_{S^\bot}$ with $u_S\in S, u_{S^\bot}\in S^\bot:=\{ v\in H:\langle v,w\rangle=0, \forall w\in S\}$. 
Then we have
\[
{\mathbb P}_C(u) = \underset{c\in C}{\arg \min} \|c-u\|
= \underset{c\in C}{\arg\min} \|c-u_S\| = {\mathbb P}_C({\mathbb P}_S(u)).
\]
In the above, the second equality holds because $\|c-u\|^2=\|c-u_S\|^2+\|u_{S^\bot}\|^2$ for every $c\in C$.
\end{proof}

Then by Lemma~\ref{LM0.2}, this problem can be solved in  two steps:
First, we solve the projection problem $\min_{x\in {\mathbb R}^n}\|\sum_{i=1}^n x_i u_i-u_0\|^2$. 
The optimal solution $\bar{x}$ has a closed form expression $\bar{x}=G^{-1}[\langle u_0, u_j\rangle]_{(1\le j\le n)}$ 
where $G=[\langle u_i, u_j\rangle]_{( 1\le i,j\le n)}$ is the positive definite Gram matrix and
$\bar{u}=\sum_{i=1}^n\bar{x}_iu_i$ is the projection of $u_0$ onto the subspace $\mbox{span}\{ u_1,...,u_n\}$.

Then the optimization problem (\ref{EQLQ}) is equivalent to 
\begin{equation}\label{EQLQ2}
\left\{\begin{array}{cl}
\underset{x\in {\mathbb R}^n}{\mbox{Minimize}}& J(x) = \frac12\big\|\sum\limits_{i=1}^n (x_i-\bar{x}_i)u_i\big\|^2\\
\mbox{subject to } & x\ge 0.
\end{array}
\right.
\end{equation}
One can reformulate $J(x) =\frac12 (x-\bar{x})^TG(x-\bar{x})$, thus we have $J'(x)=G(x-\bar{x})$.
To solve \reff{EQLQ2} with Algorithm~\ref{alg:WSM}, the iterative formula \reff{eq:lcx^k} simply becomes
\begin{equation}\label{EQQI}
x^{(k+1)} \,=\, x^{(k)} + \alpha_k [-Gx^{(k)}+G\bar{x}]^{\mathbb P},
\end{equation}
where $\alpha_k>0$ is determined by the feasiblility $x^{(k+1)}\ge 0$ and the stepsize rule
\[
J(x^{(k+1)}) - J(x^{(k)}) \le 
-\frac{\alpha_k}{4}\|[-J'(x^{(k)})]^{\mathbb P}\|^2.
\] 
In addition, denote $Q=G^{1/2}$, then the stepsize rule can be explicitly expressed as:
\begin{equation}\label{EQQI-2}
\begin{array}{c}
\alpha_k (x^{(k)}-\bar{x})^TG[-G(x^{(k)}-\bar{x})]^{\mathbb P}
+ (\alpha_k)^2\big\|Q[-G(x^{(k)}-\bar{x})]^{\mathbb P} \big\|^2\\
\le -\frac{\alpha_k}{4}\big\|[G(x^{(k)}-\bar{x})]^{\mathbb P}\big\|^2.
\end{array}
\end{equation}
When the equality holds, we obtain
\begin{equation}\label{eq:alpha_k}
\alpha_k=\frac{-
2(x^{(k)}-\bar{x})^TG[-G(x^{(k)} - \bar{x})]^{\mathbb P}-\frac{1}{4}
\|[-G(x^{(k)}-\bar{x})]^{\mathbb P}\|^2
}{\|Q[-G(x^{(k)}-\bar{x})]^{\mathbb P} \|^2}.
\end{equation}
It is an upper bound for setting $\alpha_k$ in actual computations.
By the projection property, we have
\[
-(x^{(k)}-\bar{x})^TG[-G(x^{(k)} - \bar{x})]^{\mathbb P}
\,=\, \|[-G(x^{(k)} - \bar{x})]^{\mathbb P}\|^2.
\]
Reformulate \reff{eq:alpha_k} with the above equation.
We obtain
\begin{equation}\label{EQQI-3}
\alpha_k = \frac{7}{4}\cdot\frac{
\|[-G(x^{(k)}-\bar{x})]^{\mathbb P}\|^2}{\|Q[-G(x^{(k)}-\bar{x})]^{\mathbb P} \|^2}> 0,
\end{equation}
if $d_W(x^{(k)})=[-G(x^{(k)}-\bar{x})]^{\mathbb P}\neq 0$, i.e., if $x^{(k)}$ is not a KKT point.

The fixed point expression (\ref{EQ4.3}) for a KKT point $x^*$ now becomes 
\begin{equation}\label{EQ4.8}
x^* \,=\, [x^*-\alpha J'(x^*)]^+ \,=\, [x^*-\alpha(Gx^*-G\bar{x})]^+,
\end{equation}
where $\alpha>0$ is a fixed scalar. At this fixed point, $x^*\ge 0$ and
\[
d_W(x^*) \,=\, [-J'(x^*)]^{\mathbb P} \,=\, [-Gx^*+G\bar{x}]^{\mathbb P}=0.
\]
So the fixed point iterative formula coincides with \reff{EQQI} as in Algorithm~\ref{alg:WSM}.
\begin{remark}
For given $u^{(k)}\in\Omega$,
let $u_0 = -J'(u^{(k)})$, $u_i = g_i'(u^{(k)})$ for each $i\in I_A = I_A(u^{(k)})$ and denote the Gram matrix 
$G=[\langle u_i, u_j\rangle]_{(i,j\in I_A)}$, then
the above procedure can be used to compute the projection $\mathbb{P}_{C_A}(\cdot)$   
in Algorithm~\ref{alg:WSM} at the $k$th iteration.
\end{remark}

\subsection{Optimize over other constraints}

We discuss other commonly appeared constraints in nonlinear optimization.

\noindent
$\bullet$ ({\it Two-sided bounds.})  \cite{HZ}
Two sided constraints
\[
a_i\le g_i(u)\le b_i,\quad  i=1,...,m,
\]
where $a_i<b_i$ are real numbers,
can be treated as one sided constraints by setting $\tilde{g}_i(u)\le c_i, i=1,...,2m$ where 
$\tilde{g}_i=g_i, c_i=b_i, \tilde{g}_{m+i}=-g_i, c_{m+i}=-a_i, i=1,...,m$. The assumption that 
$\{g'_i(u): i\in I_A(u)\}$ are linearly independent will be enough since the case $g_i(u)=a_i$ and 
$g_i(u)=b_i$ will never happen for all $u\in\Omega$.
 
\noindent

\noindent
$\bullet$ ({\it Mixed (in)equality constraints.}) 
When there are also equality constraints, i.e.,
$h_i(u) = 0, i\in \mc{I}_2$ as in (\ref{EQ1.0}), denote the feasible set
$\Omega = \{ u\in H : g_i(u) \le 0, i\in {\mathcal I}_1, h_j(u) = 0, j\in{\mathcal I}_2\}$. 
For each $u\in\Omega$, assume $g'_i(u), h'_j(u)$ for all $i\in I_A(u), j\in{\mathcal I}_2$ are linearly independent. 
Denote $S(u)=\mbox{span}\{ h'_j(u):j\in {\mathcal I}_2\}$,
\begin{eqnarray}
I^+_A(u) = I_A(u)\cup {\mathcal I}_2, && I^+_W(u) = I_W(u)\cup {\mathcal I}_2,\label{EQ4.15}\\
C^+_A(u) = C_A(u)\oplus S(u), && C^+_W(u) = C_W(u)\oplus S(u),\label{EQ4.16}\\
d^+_A(u) = -J'(u)-{\mathbb P}_{C^+_A(u)}(-J'(u)), && d^+_W(u)=-J'(u)-{\mathbb P}_{C^+_W(u)}(-J'(u)).\nonumber
\end{eqnarray}
Since $C^+_A(u)$ and $C^+_W(u)$ are closed and convex, Lemma~\ref{LM2} can be applied by replacing 
$C_A(u), d_A(u)$ with $C^+_A(u), d^+_A(u)$ respectively to show that
$$\langle g'_i(u), d^+_A(u)\rangle< 0,\;\forall i\in I_A(u)\setminus I_W(u),\quad
\langle g'_i(u), d^+_A(u)\rangle = 0,\;\forall i\in I_W(u),$$
$$\langle h'_j(u), d^+_A(u)\rangle = 0,\;\forall j\in{\mathcal I}_2,\;
\langle {\mathbb P}_{C^+_A(u)}(-J'(u)), d^+_A(u)\rangle=0.
$$
For each non-KKT point $u\in \Omega$, define a supporting hyperplane of $C^+_A(u)$ at the point 
${\mathbb P}_{C^+_A(u)}(-J'(u))$,
\[
{\mathcal H}^+_A(u) \,:=\, \{ v\in H : \langle v, d^+_A(u)\rangle = 0\}.
\]
It separates $C^+_A(u)$ in its lower half space $\{ v\in H : \langle v, d^+_A(u)\rangle \le 0\}$ from 
$-J'(u)$ in the upper half space $\{ v\in H : \langle v, d^+_A(u)\rangle>0\}$. 
Note that $C^+_W(u)$ is also a subset of $\mc{H}^+_A(u)$.
Indeed, $C^+_W(u)$ is a face of the convex cone $C^+_A(u)$ 
whereat the projection ${\mathbb P}_{C^+_A(u)}(-J'(u))$ of $-J'(u)$ onto $C^+_A(u)$ locates. 
It implies that the projections of $-J'(u)$ onto $C^+_A(u)$ and onto $C^+_W(u)$ are the same, i.e.,
\begin{equation}
\begin{array}{rclcl}
d^+_A(u) &=&  -J'(u)-{\mathbb P}_{C^+_A(u)}(-J'(u)) & &\\
&=& -J'(u)-{\mathbb P}_{C^+_W(u)}(-J'(u)) &= &d^+_W(u).
\end{array}
\end{equation}
Then it is clear that for each $u\in\Omega$,
\begin{eqnarray*}
d^+_A(u)=0&\Leftrightarrow& -J'(u)={\mathbb P}_{C^+_A(u)}(-J'(u))\in\, C^+_A(u)=
C_A(u)\oplus S(u)\\
&\Leftrightarrow& -J'(u)=\sum_{i\in I_A(u)} \mu_i g'_i(u)+\sum_{j\in\mc{I}_2}\lambda_j h'_j(u),
\end{eqnarray*}
where $\mu_i\ge 0$. 
This is equivalent to the KKT conditions (\ref{EQKKT1})-(\ref{EQKKT4}). Thus $u$ is a KKT point of (\ref{EQ1.0}). 
By ${\mathbb P}_{C^+_A(u)}(-J'(u))={\mathbb P}_{C^+_W(u)}(-J'(u))\in {\mathcal H}^+_A(u)$, we should also have
$$-J'(u)\,=\, \sum_{i\in I_A(u)} \mu_i g'_i(u)+\sum_{j\in \mc{I}_2}\lambda_j h'_j(u)
\,=\, \sum_{i\in I_W(u)} \mu_i g'_i(u)+\sum_{j\in \mc{I}_2}\lambda_j h'_j(u),$$
where $\mu_i=0, i\in I_A(u)\setminus I_W(u)$ since $\langle d^+_A(u), g'_i(u)\rangle<0,\;
\forall i\in I_A(u)\setminus I_W(u)$ or $\{g'_i(u): i\in I_A(u)\setminus I_W(u)\}\cap {\mathcal H}^+_A(u)=\emptyset$.

Next we modify the correction form (\ref{EQCRR1}) accordingly as
$$u(t) \,=\, u + td^+_A(u) + \sum_{i\in I_W(u)} c_i(t)g'_i(u) + \sum_{j\in\mc{I}_2}c_j(t) h'_j(u).$$
With those modifications,  Lemmas~\ref{LM1}, \ref{LM3} and \ref{LM2.5} can all be verified without any difficulty. 
Finally the Algorithm~\ref{alg:WSM} can be applied to (\ref{EQ1.0}).

\noindent
$\bullet$ ({\it Pure equality constraints.}) 
When there are only equality constraint, i.e., $\mc{I}_1=\emptyset$, 
we set $I_A(u)\equiv I_W(u)=\emptyset$ in (\ref{EQ4.15}) and $C_A(u)=C_W(u)=\{0\}$ in (\ref{EQ4.16}). 
Then it follows $C^+_A(u)=C^+_W(u)=S(u)$ and $d^+_A(u)=d^+_W(u)=-J'(u)-{\mathbb P}_{S(u)}(-J'(u))$.
From the very lase case,  our Algorithm~\ref{alg:WSM}
 is reduced to the usual gradient projection method.

\noindent
$\bullet$ ({\it Unconstrained case.})
When there is no constraint, i.e., $\mc{I}_1 = \mc{I}_2 = \emptyset$, 
we set $I_A(u)\equiv I_W(u)\equiv \emptyset$ in (\ref{EQ4.15}) and $C^+_A(u)=C^+_W(u)=\{0\}$ in (\ref{EQ4.16}).  
Then it follows $d^+_A(u)=d^+_W(u)=-J'(u)$. 
Our Algorithm~\ref{alg:WSM} automatically becomes the steepest descent method. 

\section{Numerical Examples}\label{SECEX}
\setcounter{equation}{0}
\renewcommand{\theequation}{\thesection.\arabic{equation}}

In this section, we carry out numerical experiments by applying Algorithm~\ref{alg:WSM}
to solve some classic test probelms for nonlinear constrained optimization.
The computation is implemented in {\tt MATLAB 2021a}, in a desktop with CPU 
Intel\textregistered Core\texttrademark i7-6700 and RAM 16GB.
Due to round-off error, we make some tolerance to determine the feasibility and active sets.
In iterations, we require $g_i(\bar{u})<-10^{-7}$ for each $i\in\mc{I}\setminus I_A$ to guarantee the feasibility.
For a point $u\in H$, its active set is identified by $I_A(u) = \{i\in\mc{I}: |g_i(u)|\le 10^{-5}\}$.
Its WIS $I_W$ is the set of index $i\in I_A(u)$ such that $|\langle d,g_i'(u)\rangle| \le 10^{-5}\cdot \|d\|$. 
In each of the following examples, we set $\tau = 1$ and $\varepsilon =10^{-4}$. 
Thus algorithm iterations will be terminated when $\|d\|<\varepsilon =10^{-4}$.
For neat of expression, we report our computational results with four decimals.

First, we apply our algorithm to a relatively simple but classic nonlinear constrained optimization problem.
We selected different starting points to show the efficiency and properties of our algorithm.

To better understand from the computational results how efficiency our algorithm is to identify a FAI, one should
note that an index $i\in I_A(u^{(k)})\setminus I_A(u^{(k+1)})$ is a FAI identified and removed by the algorithm 
while an index $i\in I_A(u^{(k)})\cap I_A(u^{(k+1)})$ is an active index in $I_W(u^{(k)})$ identified and kept by the algorithm.

\begin{example}\label{ex:Rosenbrock}
Consider optimizing the Rosenbrock function
\[
\left\{
\begin{array}{cl}
\underset{u\in\re^2}{\mbox{Minimize}} & J(u)=(1-u_1)^2 + 100(u_2-(u_1)^2)^2\\
\st & g_1(u)=(u_1-1)^3-u_2+1\le 0,\\
& g_2(u)=u_1+u_2-2\le 0.\\
\end{array}
\right.
\]
The optimization has global optimal value and solution:
\[
J^* = 0,\quad u^* = (1, 1).
\]
Clearly, $I_A(u^*) = \{1,2\}$.
We apply our algorithm to this probelm with different starting points 
\[
\mbox{(i).} u^{(1)} = (0,0),\quad \mbox{(ii).} u^{(1)} = (0.5, 1.5),\quad
\mbox{(iii).} u^{(1)} = (0,1).
\]
The active set of each initial point is different.

\noindent
(i). When $u^{(1)} = (0,0)$, our algorithm terminated at loop $k=9$ with $2.23$ seconds.
The computational results are reported in the following table.
\begin{center}
\begin{tabular}{|c|c|c|c|c|c|}
\hline
$k$ & $1$ & $2$ & $3$ & $4-9$\\
\hline
$u^{(k)}$ & $\bbm 0.0000\\0.0000\ebm$ &  $\bbm  0.0008\\0.0023\ebm$ & $\bbm 0.0010\\0.0030\ebm$ &
$\bbm 0.0011\\0.0033\ebm$ \\
\hline
$I_A(u^{(k)})$ & $\{1\}$ & $\{1\}$ & $\{1\}$ & $\{1\}$ \\
\hline
\end{tabular}
\end{center}
The active set for each $u^{(k)}$ remains to be $\{1\}$,
which implies that our method searched along the boundary of the feasible set with $g_1(u)=0$.

\noindent 
(ii). When $u^{(1)} = (0.5, 1.5)$, our algorithm terminated at the loop $k = 6$ with $1.57$ second.
The computational results are reported in the following table.
\begin{center}
\begin{tabular}{|c|c|c|c|c|c|c|}
\hline
$k$ & $1$ & $2$ & $3$ & $4$ & $5$ & $6$\\
\hline
$u^{(k)}$ & $\bbm 0.5000\\1.5000\ebm$ &  $\bbm  0.9893\\1.0107\ebm$ & $\bbm 0.9986\\1.10014\ebm$ &
$\bbm 0.9998\\1.0002\ebm$ &  $\bbm 1.0000\\1.0000\ebm$ & $\bbm 1.0000\\1.0000\ebm$\\
\hline
$I_A(u^{(k)})$ & $\{2\}$ & $\{2\}$ & $\{2\}$ & $\{2\}$ & $\{2\}$ & $\{1,2\}$\\
\hline
\end{tabular}
\end{center}
Our algorithm converged to the true global optimal solution,
following the boundary of feasible set with $g_2(u)=0$.

\noindent
(iii). When $u^{(1)} = (0,1)$, our algorithm terminated at loop $k=20$ with $5.39$ seconds.
The computational results are reported in the following table.
\begin{center}
\begin{tabular}{|c|c|c|c|c|c|c|}
\hline
$k$ & $1$ & $2$ & $3$ & $4$ & $5$ & $6$\\
\hline
$u^{(k)}$ & $\bbm 0.0000\\1.0000\ebm$ &  $\bbm  0.0078\\0.2188\ebm$ & $\bbm 0.0130\\0.1333\ebm$ &
$\bbm 0.0182\\0.0813\ebm$ &  $\bbm 0.0207\\0.0655\ebm$ & $\bbm 0.0210\\0.0639\ebm$\\
\hline
$I_A(u^{(k)})$ & $\emptyset$ & $\emptyset$ & $\emptyset$ &$\emptyset$ & $\emptyset$ & $\emptyset$\\
\hline
$k$ & $7$ & $8$ & $9$ & $10$ & $11$ & $12$\\
\hline
$u^{(k)}$ & $\bbm  0.0212\\0.0631\ebm$ &  $\bbm 0.0213\\0.0627\ebm$ & $\bbm 0.0213\\0.0627\ebm$ &
$\bbm 0.0213\\0.0626\ebm$ &  $\bbm 0.0213\\0.0626\ebm$ & $\bbm 0.0075\\0.0222\ebm$\\
\hline
$I_A(u^{(k)})$ & $\emptyset$ &$\emptyset$ & $\emptyset$ & $\emptyset$ &$\emptyset$ & $\{1\}$\\
\hline
$k$ & $13$ & $14$ & $15$ & $16$ & $17$ & $18-20$ \\
\hline
$u^{(k)}$ & $\bbm  0.0017\\0.0051\ebm$ &  $\bbm 0.0013\\0.0039\ebm$ & $\bbm0.0012\\0.0035\ebm$ &
$\bbm 0.0011\\0.0034\ebm$ &  $\bbm 0.0011\\0.0033\ebm$ &  $\bbm 0.0011\\0.0033\ebm$ \\
\hline
$I_A(u^{(k)})$ & $\{1\}$ & $\{1\}$ &  $\{1\}$ &  $\{1\}$ & $\{1\}$ & $\{1\}$ \\
\hline
\end{tabular}
\end{center}
\begin{figure}[htb]
\epsfig{file=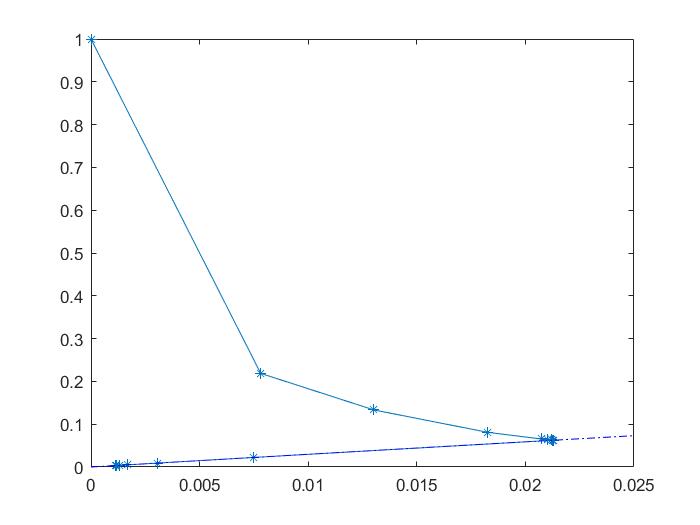,height=2.4in,width=4.6in}
\caption{The computational results of Example~\ref{ex:Rosenbrock} with the initial point $(0,1)$.
The star symbol represents $u^{(k)}$ and the dash-dotted curve represents the boundary given by $g_1(u)=0$.}
\label{fig:Rosen01}
\end{figure}
Since $I_A(u^{(1)})=\emptyset$, our algorithm made first 11 searches inside the feasible set.
It went approaching to the boundary of the feasible set determined by $g_1(u)=0$, 
and converged to the terminated point along the boundary from the loop $k=12$.
We plot these $u^{(k)}$ in Figure~\ref{fig:Rosen01} with the star symbol,
where the dash-dotted line represents the boundary of the feasible set given by $g_1(u)=(u_1-1)^3-u_2+1=0$.
From the figure, we can see the computed points first go approach to the boundary and then converge to 
the terminated point along the boundary.

To check the nature of the point $(0.0011, 0.0033)$ to which our algorithm converges with the initial guess at $(0, 0)$ or $(0,1)$, 
we plot the objective contours and the active constraint curve $g_1(u)=(u_1-1)^3-u_2+1=0$ near the point $(0.0011, 0.0033)$. 
The graph shows that the point $(0.0011, 0.0033)$ is actually a numerical local minimum point.
\begin{figure}[htb]
\epsfig{file=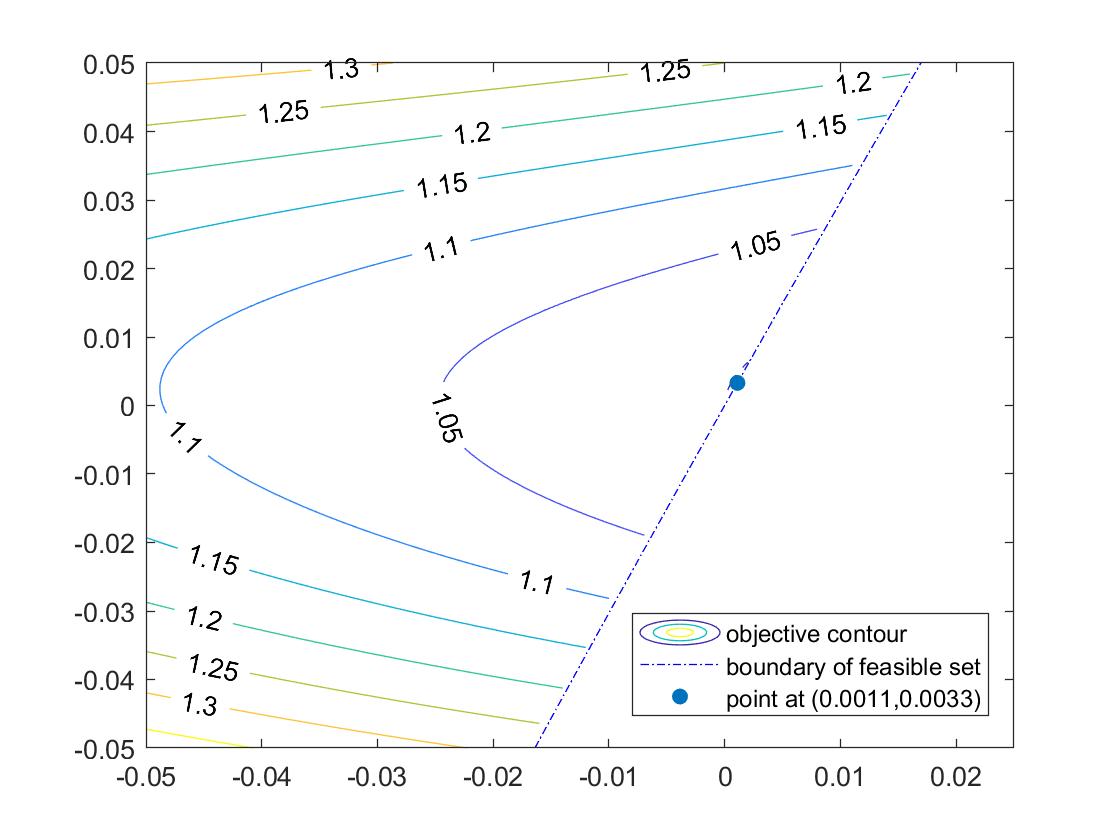,height=2.4in,width=4.6in}
\caption{The objective contours and the active constraint curve $g_1(u)=0$ near the point $(0.0011, 0.0033)$ in Example~\ref{ex:Rosenbrock}.}
\label{fig:Rosen01-ct}
\end{figure}

\end{example}
\pagebreak

\begin{example}\label{ex:Rosenbrock_disk}
Then we consider optimizing the Rosenbrock function with more complicated constraints.
\[
\left\{
\begin{array}{cl}
\underset{u\in\re^2}{\mbox{Minimize}} & J(u)=(1-u_1)^2 + 100(u_2-(u_1)^2)^2\\
\st & g_1(u)=u_1^2 + u_2^2-2\le 0,\\
& g_2(u)=0.16 - (u_1-1)^2 - u_2^2\le 0,\\
& g_3(u)=1 - u_1^2 - (u_2-2)^2\le 0.
\end{array}
\right.
\]
The above optimization problem has global optimal value and solution:
\[
J^* = 0,\quad u^* = (1, 1).
\]
Similar to Example~\ref{ex:Rosenbrock}, we consider different starting points.

\noindent
(i). When $u^{(1)} = (1,-1)$, we have $I_A(u^{(1)}) = \{2\}$. 
Our algorithm terminated at loop $k = 921$ with $160.48$ seconds.
Interestingly, we got
\[
u^{(2)} = \left[\begin{array}{r}0.6094\\-0.8047\end{array}\right],\, I(u^{(2)}) = \emptyset;\quad
u^{(921)} = \bbm 0.9999\\0.9998\ebm,\, I(u^{(921)}) = \emptyset.
\]
It implies that our algorithm identified FAI $\{2\}$ with a single loop and converged to the global optimizer 
from the interior of the feasible set.
To better illustrate the advantage of our algorithm, we plot computed points in each iteration in Figure~\ref{fig:Rosen02}.
In the figure, each $u^{(k)}$ is plotted with the star symbol and the boundary given by $g_1(u)=u_1^2+u_2^2-2=0$ 
is plotted by the dash-dotted curve.
\begin{figure}[htb]
\epsfig{file=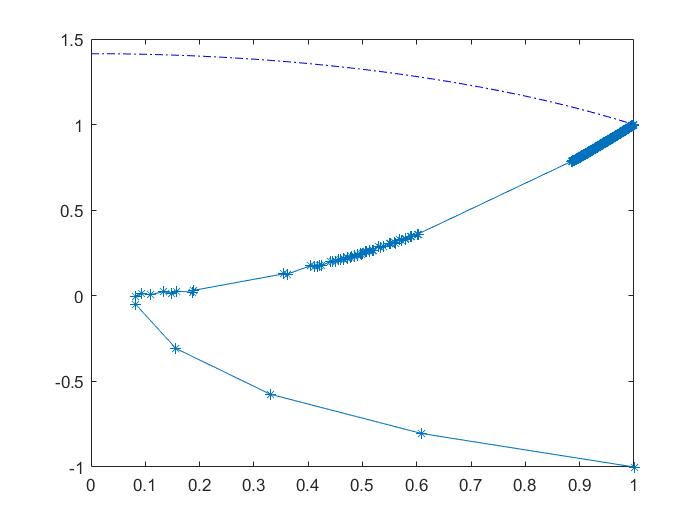,height=2.4in,width=4.6in}
\caption{The computational results of Example~\ref{ex:Rosenbrock_disk} with the initial point $(1,-1)$.
The star symbol represents $u^{(k)}$ and the dash-dotted line represents the boundary given by $g_1(u)=0$.}
\label{fig:Rosen02}
\end{figure}

\noindent
(ii). When $u^{(1)} = \left(\frac{5}{4},\frac{\sqrt{7}}{4}\right)$, we have $I(u^{(1)}) = \{1,3\}$.
Our algorithm terminated at loop $k = 616$ with $106.54$ seconds.
Interestingly, we got
\[
u^{(2)} = \left[\begin{array}{r}1.0298\\0.7494\end{array}\right],\, I(u^{(2)}) = \{1\};\quad
u^{(3)} = \bbm 0.9672\\0.7798\ebm,\, I(u^{(3)}) = \emptyset.
\]
\[
u^{(616)} = \bbm 0.9999\\0.9998\ebm,\quad I(u^{(616)}) = \emptyset.
\]
It shows our algorithm can efficiently identify the FAI and make search in the interior of the feasible set.
Similar to case (i), we plotted the computational results in Figure~\ref{fig:Rosen03}.
\begin{figure}[htb]
\epsfig{file=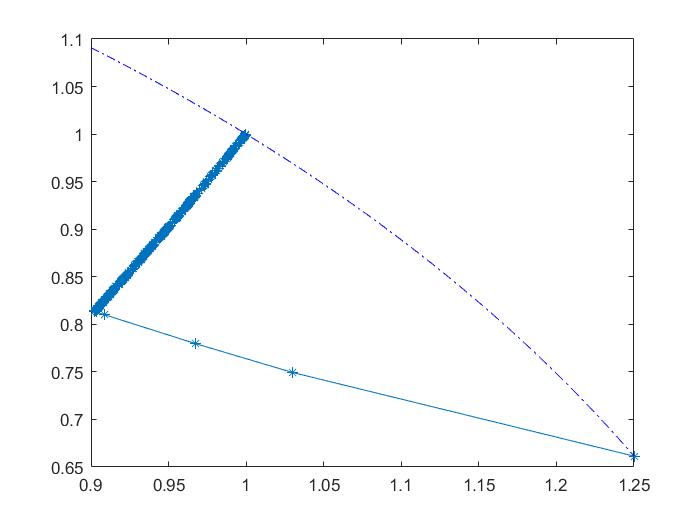,height=2.4in,width=4.6in}
\caption{The computational results of Example~\ref{ex:Rosenbrock_disk} with the initial point 
$\left(\frac{5}{4},\frac{\sqrt{7}}{4}\right)$.
The star symbol represents $u^{(k)}$ and the dash-dotted line represents the boundary given by $g_1(u)=0$.}
\label{fig:Rosen03}
\end{figure}
In the figure, each $u^{(k)}$ is plotted with the star symbol and the boundary given by $g_1(u)=u_1^2+u_2^2-2=0$ 
is plotted by the dash-dotted curve.
\end{example}

\begin{example}[Mishara's Bird]\label{MishraBird}
Consider the optimization problem
\[
\left\{
\begin{array}{cl}
\underset{u\in\re^2}{\mbox{Minimize}} & 
J(u)=\sin(u_2)e^{[(1-\cos u_1)^2]} + \cos(u_1)e^{[(1-\sin u_2)^2]} + (u_1 -u_2)^2\\
\st & g_1(u)=(u_1+5)^2+(u_2+5)^2\le 25,\\
& g_2(u)=u_1+1\le 0,\, g_3(u)=-9-u_1\le 0,\\
& g_4(u)=u_2\le 0,\, g_5(u)=-8-u_2\le 0.
\end{array}
\right.
\]
The optimization has global optimal value and solution:
\[
J^* = -106.7645,\quad u^* = (-3.1302, -1.5821).
\]
In particular, the optimal solution is an interior point.
We apply our algorithm to different initial points.

\noindent
(i). When $u^{(1)} = (-1,-8)$, $I_A(u^{(1)}) = \{1,2,5\}$.
The algorithm terminated at the loop $k = 17$ with $3.90$ seconds.
Interestingly, we got
\[
u^{(2)} = \bbm -2.8431\\-8.0000\ebm,\, I_A(u^{(2)}) = \{5\},\quad
u^{(3)} = \bbm -3.1473\\-7.8206\ebm,\, I_A(u^{(3)}) = \emptyset;
\]
\[
u^{(18)} = \bbm  -3.1757\\-7.8198\ebm, \quad I_A(u^{(18)}) = \emptyset.
\]

\noindent
(ii) When $u^{(1)} = (-5,0) $, $I_A(u^{(1)}) = \{1,4\}$.
The algorithm terminated at the loop $k = 16$ with $2.95$ seconds.
Interestingly, we got
\[
u^{(2)} = \bbm -3.4242\\-2.6410\ebm,\, I_A(u^{(2)}) = \{1,4\},\quad
u^{(3)} = \bbm-2.9783\\-1.6777\ebm,\, I_A(u^{(3)}) = \emptyset;
\]
\[
u^{(18)} = \bbm  -3.1302\\-1.5821\ebm, \quad I_A(u^{(18)}) = \emptyset.
\]
Both cases show that our algorithm can efficiently identify FAIs.

To check the nature of the point $(-3.1757,  -7.8198)$ to which our algorithm converges with the 
initial guess at $(-1, -8)$, since there is no constraint active nearby, we just plot the objective contours near 
the point $(-3.1757, -7.8198)$ in Figure~\ref{fig:Mish-Bird-ct}. The graph shows that the point $(-3.1757, -7.8198)$
is actually a numerical local minimum point.

\begin{figure}[htb]
\epsfig{file=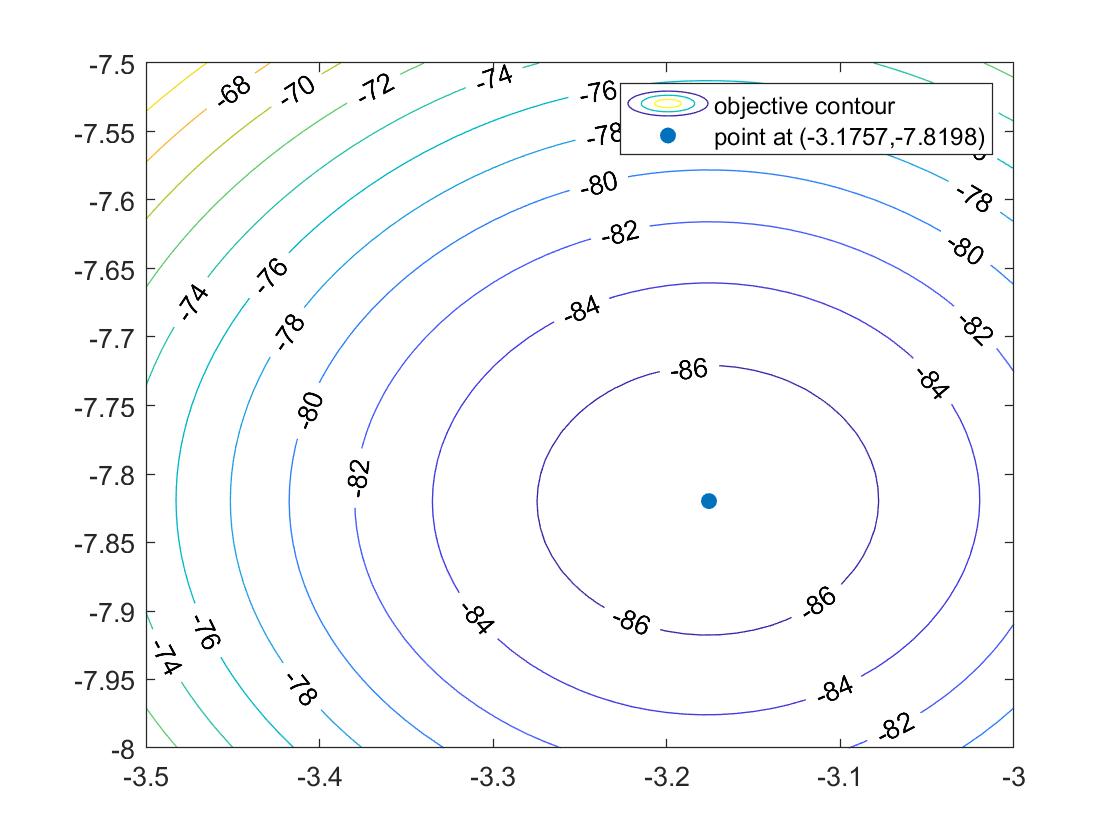,height=2.4in,width=4.6in}
\caption{The objective contours near the interior point $(-3.1757, -7.8198)$ in Example~\ref{MishraBird}.}
\label{fig:Mish-Bird-ct}
\end{figure}
\end{example}

\begin{example}[Gomez and Levy function]\label{ex:Gomes}
Consider the optimization problem
\[\left\{
\begin{array}{cl}
\underset{u\in\re^2}{\mbox{Minimize}} & J(u)=4u_1^2-2.1u_1^4+\frac{1}{3}u_1^6+u_1u_2-4u_2^2+4u_2^4\\
\st & g_1(u)=-\sin(4\pi x)+2\sin^2(2\pi u_2)-1.5\le 0,\\
& g_2(u)=-1-u_1\le 0,\, g_3(u)=u_1-0.75\le 0,\\
& g_4(u)=-1-u_2\le 0,\, g_5(u)=u_2-1\le 0,\\
& g_6(u)=u_2-u_1\le 0,\, g_7(u)=u_1u_2\le 0.
\end{array}
\right.\]
It has global optimal value and solution:
\[
J^* = -1.0316,\quad u^* = (0.0898, -0.7126).
\]
When $u^{(1)} = (-1,-1)$, $I_A(u^{(1)}) = \{2,4,6,7\}$.
The algorithm terminated at the loop $k = 11$ with $ 2.61$ seconds.
The computational results are reported in the following table.
\begin{center}
\begin{tabular}{|c|c|c|c|c|c|}
\hline
$k$ & $1$ & $2$ & $3$ & $4$ & $5$\\
\hline
$u^{(k)}$ & $\bbm -1.0000\\-1.0000\ebm$ &  $\bbm  -0.8188\\-0.8187\ebm$ & $\bbm -0.3820\\-0.4375\ebm$ &
$\bbm -0.0018\\-0.6597\ebm$ &  $\bbm 0.0403\\-0.7023\ebm$\\
\hline
$I_A(u^{(k)})$ & $\{2,4,6,7\}$ & $\{6,7\}$ & $\{7\}$ & $\{7\}$ & $\emptyset$\\
\hline
$k$ & $6$ & $7$ & $8$ & $9$ & $10-11$\\
\hline
$u^{(k)}$ & $\bbm 0.0641\\-0.7096\ebm$ &  $\bbm  0.0890\\-0.7126\ebm$ & $\bbm 0.0898\\-0.7126\ebm$ &
$\bbm 0.0898\\-0.7127\ebm$ &  $\bbm 0.0898\\-0.7127\ebm$\\
\hline
$I_A(u^{(k)})$ & $\emptyset$ &$\emptyset$ &$\emptyset$ & $\emptyset$ & $\emptyset$\\
\hline
\end{tabular}
\end{center}
It shows that our algorithm can efficiently identify FAIs.
\end{example}

\section{Comparision of WSM and ASMs}\label{SEC5}
\setcounter{equation}{0}
\renewcommand{\theequation}{\thesection.\arabic{equation}}

Next we conclude this paper by comparing WSM with ASMs in the literature. 

An efficient and stable numerical algorithm for nonlinear inequality constrained minimization is 
expected to have a monotone descent property. It is usually difficult to achieve with nonlinear inequality constraints involved.
An additional correction process is likely to be constructed to maintain the feasibilities.
Interestingly, if a correction process is superlinear, the descent search direction will dominate the correction terms, 
then the corresponding algorithm can regain the monotone descent property. 
But the correction process depends on the index set used to form its basis  and such an index set should be able to update dynamically. 

Our ASM, Algorithm~\ref{alg:WSM} possesses all those nice properties.
The key is that we use WIS $I_W(u)$ to form a correction basis, which excludes FAIs from $I_A(u)$. 
It is clear that once there is an $i\in I_A(u)\setminus I_W(u)$ used in the correction basis, we have
$\langle d_A(u), g'_i(u)\rangle<0$, where $d_A(u)=d_W(u)$ is the search direction. 
Then the correction process will not be superlinear and the stepsize rule Lemma~\ref{LM2.5} fails to hold.
 
One may ask can we change the search direction $d(u)=d_W(u)$ or the working index set $I_W(u)$
while preserving the monotone descent property and the dynamics of updating $I_A(u)$
in ASMs?

For a feasible point $u\in \Omega$, 
let $I_B(u)$ satisfying $I_W(u)\subset I_B(u)\subseteq I_A(u)$ be an AIS. 
Such $I_B(u)$ may represent some strategy to get rid of some FAIs in an ASM.
Let $C_B(u)$ be the polyhedral cone formed by edges $g_i'(u),\, i\in I_B(u)$, 
and denote 
\[
d_B(u) := -J'(u)-{\mathbb P}_{C_B(u)}(-J'(u)).
\] 
Then similar to (\ref{EQDA=DW}) and by $I_W(u)\subset I_B(u)\subseteq I_A(u)$, we have 
\begin{equation}\label{EQPA=PW}
{\mathbb P}_{C_W(u)}(-J'(u)) \,=\, {\mathbb P}_{C_B(u)}(-J'(u)) \,=\, {\mathbb P}_{C_A(u)}(-J'(u)),
\end{equation}
\begin{equation}\label{EQ2.6}
d_A(u) \quad =\quad d_B(u) \quad =\quad d_W(u).
\end{equation}
In this case, the correction process
 $$u(t)=u+td_B(u)+\sum_{i\in I_B(u)}c_i(t)g'_i(u)$$ 
will satisfy more active constraints including some FAIs.
However, the strict inequality for every $i\in I_A(u)\setminus I_W(u)$, i.e.,
\[
\langle g_i'(u), d_A(u)\rangle<0,\quad \forall i\in I_A(u)\setminus I_W(u)
\]
implies that we cannot get $c_i(t)=o(t)$. Thus the correction term will not be superlinear.
In other words,  enlarging $I_W(u)$ to $I_B(u)$ will lose the superlinear property.
On the other hand, since 
\[
\langle g'_i(u), d_W(u)\rangle=0,\quad  \forall i\in I_W(u),
\] 
the WIS $I_W(u)$ cannot be further reduced to get a feasible iterative $u(t)$,
unless the second derivative terms $g_i''(u), i\in I_W(u)$ are used for further analysis, e.g., 
to be negative definite. In other words, $I_W(u)$ is an {\it optimal} WIS in the sense of active index selection
 to have a superlinear correction, with which $d_W(u)$ is its correctable steepest descent search direction. 
Our WSM identifies FAIs and dynamically update AIS $I_A(u)$ in each algorithm iteration.

Next from the correction process (\ref{EQCRR1}),  we know that $d_W(u)$ is the search direction and $I_W(u)$, 
where FAIs excluded, is the active index set used to form the correction basis $\{ g_i'(u) :i\in I_W(u)\}$. 
Our analysis shows that it is this property $d_W(u)\bot g'_i(u)$ for all $i\in I_W(u)$ that leads to
the superlinear correction, a critical property in establishing the stepsize rule for the algorithm and its convergence analysis [\ref{Zhou0}].
For ASMs in the literature, $d_S(u)$ or $d_A(u)$ is used as a search direction and $I_A(u)$ or $I_A(u)\setminus \{\bar i\}$, 
where ${\bar i}$ is the active index with the moset negative KKT multiplier compunent [\ref{NW}] as mentioned in the introduction, 
is used to form a basis for the correction process. Firstly, since $d_S(u)\bot g'_i(u),\;\forall i\in I_A(u)$,  
the correction is still superlinear but $d_S(u)=0$ cannot guarantee that $u$ is a KKT point. 
Secondly, since $\langle d_A(u), g'_i(u)\rangle<0 ,\;\forall i\in I_A(u)\setminus I_W(u)\subset I_A(u)$, 
the correction cannot be superlinear. It will cause difficulty in establishing a stepsize rule for the algorithm and 
in its convergence analysis as well. Thirdly, to form a correction basis, if one uses the active index set $I_A(u)$,  
it will force the monotone inclusion relation $I_A(u^{(k)})\subset I_A(u^{(k+1)})$, i.e., 
the active index set will be monotonely increased. And if one uses the active index set $I_A(u)\setminus \{\bar i\}$, 
it will put a natural lower bound for the number of iterations to reach an optimal solution. 
In either case, it lacks of dynamics to adaptively update the active index set in algorithm iterations. 

Finally to show CSDD as defined in \reff{EQCSDD} is an {\em optimal correctable descent direction}, 
we compare it to any other correctable direction in the next lemma.
\begin{lemma}\label{LM2.4}
Let $d$ be any correctable direction, i.e., $d\bot g'_i(u),\,\forall  i\in I_W(u)$. 
Write the orthogonal decomposition  $-J'(u)=d(u)+(-J'_{d^\bot}(u))$ with 
$d(u)=-J'_d(u)\in\mbox{\rm span}\{d\}$ and $-J'_{d^\bot}(u)\in d^\bot=\{ v\in H: \langle v, d\rangle=0\}$. 
Then we have
\begin{equation}\label{EQ2.12}
\langle J'(u), d_W(u)\rangle\le\langle J'(u), d(u)\rangle\le 0\quad \mbox{and}\quad \|d(u)\|\le\|d_W(u)\|,
\end{equation}
where the first or third ``$=$'' holds if and only if $d(u)=d_W(u)$ and
the second ``$=$'' holds if and only if $d(u)=0$.
\end{lemma}
\begin{proof}
Since $d\bot g'_i(u), i\in I_W(u)$ or $C_W\subset S_W\subset d^\bot$. By Lemma~\ref{LM2}(4) 
and (\ref{EQDA=DW}), we have
\begin{equation}\label{EQ2.16}
\|{\mathbb P}_{C_W}(-J'(u))\|\le\|{\mathbb P}_{S_W}(-J'(u))\|\le\|-J'_{d^\bot}(u)\|.
\end{equation}
Thus one can formulate
\begin{eqnarray}
-\|d_W(u)\|^2&=&\langle J'(u), d_W(u)\rangle\\
&=&
\langle J'(u), -J'(u)-{\mathbb P}_{C_W}(-J'(u))\rangle\nonumber\\
&=&\langle J'(u), d(u)-J'_{d^\bot}(u)-{\mathbb P}_{C_W}(-J'(u))\rangle\nonumber\\
&=&\langle J'(u), d(u)\rangle-\|-J'_{d^\bot}(u)\|^2+\|{\mathbb P}_{C_W}(-J'(u))\|^2\nonumber\\
&\le&\langle J'(u), d(u)\rangle=-\|d(u)\|^2\le 0.\label{EQ2.13},
\end{eqnarray}
i.e., (\ref{EQ2.12}) holds. Denote $[S_W]^\bot=\{ v\in d^\bot : v\bot S_W\}$.
We have a decomposition $d^\bot=S_W\oplus [S_W]^\bot$. If we write
$$-J'_{d^\bot}(u) \,=\, {\mathbb P}_{S_W}(-J'_{d^\bot}(u))+{\mathbb P}_{[S_W]^\bot(u)}(-J'_{d^\bot}(u)),$$
 then
$\,\|-J'_{d^\bot}(u)\|^2=\|{\mathbb P}_{S_W}(-J'_{d^\bot}(u))\|^2+\|{\mathbb P}_{[S_W]^\bot}(-J'_{d^\bot}(u))\|^2$.
When the first or third ``$=$'' in (\ref{EQ2.12}) holds, from (\ref{EQ2.16}) and (\ref{EQ2.13}), we must have 
\begin{equation}\label{EQ2.15}
\|-J'_{d^\bot}(u)\|=\|{\mathbb P}_{C_W}(-J'(u))\|=\|{\mathbb P}_{S_W}(-J'(u))\|
\end{equation}
and 
$\,\|{\mathbb P}_{[S_W]^\bot}(-J'_{d^\bot}(u))\|=0.\,$
It leads to $-J'_{d^\bot}(u)\in S_W$.
Since $C_W\subset S_W\subset d^\bot$, by the uniqueness of the projection and (\ref{EQ2.15}), we get
$$\,-J'_{d^\bot}(u)={\mathbb P}_{S_W}(-J'(u))={\mathbb P}_{C_W}(-J'(u)).\,$$ 
Then
$$d_W(u)=-J'(u)-{\mathbb P}_{C_W}(-J'(u))=-J'(u)-J'_{d^\bot}(u)=d(u).$$
When the second ``$=$'' holds in (\ref{EQ2.12}), by (\ref{EQ2.13}), we have $d(u)=0$.
\end{proof}

\end{document}